\documentclass{amsart}

\usepackage{pgfplots}

\usepackage[active]{srcltx}

\usepackage{graphicx}

\usepackage{latexsym}
\usepackage{amsmath,amsthm}
\usepackage{amsfonts}
\usepackage[psamsfonts]{amssymb}
\usepackage{enumerate}
\usepackage{url}

\RequirePackage[OT1]{fontenc}

\usepackage{calrsfs}

\usepackage{hyperref}

\theoremstyle{plain} 
\newtheorem{theorem}{Theorem}[section]
\newtheorem{corollary}[theorem]{Corollary}
\newtheorem{lemma}[theorem]{Lemma}
\newtheorem{proposition}[theorem]{Proposition}

\theoremstyle{definition} 

\theoremstyle{definition} 

\theoremstyle{remark} 

\theoremstyle{remark} 
\newtheorem{remark}[theorem]{Remark}
\newtheorem*{remark*}{Remark}
\numberwithin{equation}{section}

\makeatletter
  \renewcommand\section{\@startsection {section}{1}{\z@}%
                                   {-\bigskipamount}%
                                   {\medskipamount}%
                                   {\large\bfseries
                                   \raggedright}}

  \renewcommand\subsection{\@startsection {subsection}{2}{\z@}%
                                   {-\medskipamount}%
                                   {\smallskipamount}%
                                   {\bfseries
                                   \raggedright}}
\makeatother

\renewcommand{\bar}{\overline}

\newcommand{\ffrown}{\text{\raisebox{3pt}[0pt][0pt]{$\frown$}}}
\renewcommand{\O}{\underset{\ffrown}{<}}
\newcommand{\OG}{\underset{\ffrown}{>}}

\newcommand{\ssim}{\text{\raisebox{2.5pt}[0pt][0pt]{$\sim$}}}
\newcommand{\lsim}{\underset{\ssim}{<}}
\newcommand{\gsim}{\underset{\ssim}{>}}

\renewcommand{\gg}{>\kern-2pt>}
\renewcommand{\ll}{<\kern-2pt<}

\newcommand{\dd}{\partial}
\renewcommand{\dd}{\operatorname{d}\!}

\renewcommand{\le}{\leqslant}
\renewcommand{\ge}{\geqslant}

\newcommand{\de}{\delta}

\newcommand{\vpi}{\varphi}

\newcommand{\LD}{\mathcal{L}\!\mathcal{D}}
\renewcommand{\LD}{\mathcal{L}{\kern -1.9pt}\mathcal{D}}
\renewcommand{\LD}{\mathcal{D}}
\renewcommand{\LD}{\mathcal{L}{\kern -4pt}\mathcal{C}}
\renewcommand{\LD}{\mathcal{R}{\kern -3pt}\mathcal{C}}

\newcommand{\ii}[1]{\mathrm{I}\!\left\{#1\right\}}

\newcommand{\R}{{\mathbb{R}}}
\newcommand{\C}{\mathbb{C}}

\newcommand{\G}[1]{\mathcal{G}^{#1}_+}
\renewcommand{\G}{\Psi}

\newcommand{\tr}{{\tilde{r}}}

\newcommand{\vp}{\varepsilon}

\begin{document}


\title{Exact bounds on the inverse Mills ratio and its derivatives}


\author{Iosif Pinelis}

\address{Department of Mathematical Sciences\\
Michigan Technological University\\
Houghton, Michigan 49931, USA\\
E-mail: ipinelis@mtu.edu}

\keywords{
Inverse Mills ratio, 
complementary error function, 
functions of a complex variable, 
rate of growth of functions, 
inequalities for integrals, monotonic functions, entire and meromorphic functions, boundary behavior, inequalities in approximation, best constants, approximations to distributions}

\subjclass[2010]{Primary 30A10, 30D40, 	33B20, 26D15; secondary 26A48, 41A17, 41A44, 60E15, 62E17}


\begin{abstract}
The inverse Mills ratio is $R:=\vpi/\G$, where $\vpi$ and $\G$ are, respectively, the probability density function and the tail function of the standard normal distribution.  
Exact bounds on $R(z)$ for complex $z$ with $\Re z\ge0$ are obtained, which then yield logarithmically exact bounds on high-order derivatives of $R$. The main idea of the proof is a non-asymptotic version of the so-called stationary-phase method. 
\end{abstract}

\maketitle

\tableofcontents

\section{Introduction, summary, and discussion}\label{intro}




The inverse Mills ratio $R$ is defined by the formula 
\begin{equation}\label{eq:R}
	R:=\frac{\vpi}{\G}, 
\end{equation}
where $\vpi$ and $\G$ are, respectively, the probability density function and the tail function of the standard normal distribution, so that $\vpi(x)=\frac1{\sqrt{2\pi}}\,e^{-x^2/2}$ and $\G(x)=\int_x^\infty\vpi(u)\dd u$ 
for all real $x$. 
These expressions for $\vpi$ and $\G$ in fact define entire functions on the complex plane $\C$, if the upper limit of the integral is still understood as the point $\infty=\infty+0i$ on the extended real axis. 
One may note that $\G$ is a rescaled version of the complementary error function: $\G(z)=\operatorname{erfc}(z/\sqrt2\,)/2$ for all $z\in\C$.     

\begin{theorem}\label{th:}\ 
\begin{enumerate}[(I)]
	\item \label{G ne 0} The function $\G$ has no zeros on the right half-plane 
\begin{equation}
	H_+:=
\{z\in\C\colon\Re z\ge0\}. 
\end{equation}
So, the function $S\colon H_+\to\C$ defined   
by the formula 
\begin{equation}
	S(z):=\frac{R(z)}{z+\sqrt{2/\pi}}  
\end{equation}
for $z\in H_+$ is holomorphic on the interior of $H_+$. 
	\item \label{ineq} One has 
\begin{equation}\label{eq:ineq}
	1>|S(z)|>|S(iy_*)|=0.686\ldots\quad\text{for all }z\in H_+\setminus\{0,iy_*,-iy_*\}, 
\end{equation}
where $y_*=1.6267\ldots$ is the only minimizer of $|S(iy)|$ in $y\in(0,\infty)$. 
	\item \label{exact} It obviously follows that the lower bound $|S(iy_*)|$ on $|S(z)|$ in \eqref{eq:ineq} is exact. The upper bound $1$ on $|S(z)|$ in \eqref{eq:ineq} is also exact, in following rather strong sense: 	
\begin{equation}\label{eq:lim} 
	S(0)=1=\lim_{H_+\ni z\to\infty}S(z). 
\end{equation}
\end{enumerate}
\end{theorem}

In view of the maximum modulus principle (see e.g.\ Section~3.4 of \cite{ahlfors79
}) applied to the function $S$ and its reciprocal $1/S$, Theorem~\ref{th:} is an immediate corollary of 

\begin{lemma}\label{lem:}\ 
\begin{enumerate}[(i)] 
		\item \label{conj} For all $z\in H_+$ one has $\vpi(\bar z)=\bar{\vpi(z)}$ and $\G(\bar z)=\bar{\G(z)}$, where, as usual, the bar $\bar{\phantom{x}}$ denotes the complex conjugation. 
		\item \label{G ne 0,lem} The function $\G$ has no zeros on $H_+$. Moreover, $\Re R(z)>0$ and $\Im R(z)$ equals $\Im z$ in sign, for any $z\in H_+$. 
%
		\item \label{lim} Equalities \eqref{eq:lim} hold. 
	\item \label{extr} There is a (necessarily unique) point $y_*\in(0,\infty)$ such that $|S(iy)|$ (strictly) decreases from $S(0)=1$ to $|S(iy_*)|$ as $y$ increases from $0$ to $y_*$, and $|S(iy)|$ (strictly) increases from $|S(iy_*)|$ to $1$ as $y$ increases from $y_*$ to $\infty$. In fact, $y_*=1.6267\ldots$ and $|S(iy_*)|=0.686\ldots$. 
\end{enumerate}
\end{lemma}

All necessary proofs are deferred to Section~\ref{proofs}. 

The main idea of the proof is a non-asymptotic version of the so-called stationary-phase method, which latter is described and used for asymptotics e.g.\ in \cite{guillemin-sternberg}. The mentioned version of the method is given by formulas~\eqref{eq:A-iB}--\eqref{eq:A,B}, which provide comparatively easy to analyze integral expressions for the real and imaginary parts of the Mills ratio $\G(z)/\vpi(z)$ for $z$ with $\Re z>0$. 

\begin{remark}\label{rem:}
It also follows from Lemma~\ref{lem:} by the maximum modulus principle (or, slightly more immediately, by the minimum modulus principle) that for each $x\in[0,\infty)$ the minimum $\min\{|S(x+iy)|\colon y\in\R\}$ is attained and is (strictly) increasing in $x\in[0,\infty)$, from $0.686\ldots$ to $1$. 
Of course, instead of the family of vertical straight lines $\big(\{x+iy\colon y\in\R\}\big)_{x\in[0,\infty)}$, one can take here any other appropriate family of curves, e.g.\ the family $\big(\{x+\vpi(y)+iy\colon y\in\R\}\big)_{x\in[0,\infty)}$. 
\end{remark}

In particular, Theorem~\ref{th:} implies that $\max_{x\ge0}S(x)=1$. One can also find $\min_{x\ge0}S(x)$. More specifically, one has the following proposition, complementing Theorem~\ref{th:}. 

\begin{proposition}\label{prop:x>0}
There is a (necessarily unique) point $x_*\in(0,\infty)$ such that $S(x)$ decreases from $S(0)=1$ to $S(x_*)$ as $x$ increases from $0$ to $x_*$, and $S(x)$ increases back to $1$ as $x$ increases from $x_*$ to $\infty$. In fact, $x_*=(\pi-1)\sqrt{2/\pi} = 1.7087\ldots$ and $S(x_*)=0.844\ldots$. 
\end{proposition} 

The Mills ratio of the real argument, including related bounds and monotonicity patterns, has been studied very extensively; see e.g.\ \cite{shenton,baricz08} and further references therein.

Theorem~\ref{th:}, Lemma~\ref{lem:}, Remark~\ref{rem:}, and Proposition~\ref{prop:x>0} are illustrated in Figure~\ref{fig:}. 

\pgfplotsset{width=5.5cm
}

\begin{figure}[htbp]
	\centering
\begin{tikzpicture}
\begin{axis}[footnotesize,
view/h=40,
xlabel={\footnotesize $x$}, ylabel={\footnotesize $y$},  zlabel={\footnotesize $|S(z)|\phantom{nnnn}$}, 
every axis z label/.style=
{at={(zticklabel cs:0.5)},rotate=0,anchor=center},
every axis x label/.style=
{at={(xticklabel cs:0.5,3pt)},rotate=0,anchor=center},
every axis y label/.style=
{at={(yticklabel cs:0.5,-1pt)},rotate=0,anchor=center},
]
\addplot3[surf,mesh/ordering=y varies,mesh/rows=41
]
table {abs.dat};
\end{axis}
\end{tikzpicture}
\\ \ 
\\ 
\begin{tikzpicture}
\begin{axis}[footnotesize,
view/h=40,
xlabel={\footnotesize $x$}, ylabel={\footnotesize $y$},  zlabel={\footnotesize $\Re S(z)\phantom{nnnn}$}, 
every axis z label/.style=
{at={(zticklabel cs:0.5)},rotate=0,anchor=center},
every axis x label/.style=
{at={(xticklabel cs:0.5,3pt)},rotate=0,anchor=center},
every axis y label/.style=
{at={(yticklabel cs:0.5,-1pt)},rotate=0,anchor=center},
]
\addplot3[surf,mesh/ordering=y varies,mesh/rows=41
]
table {re.dat};
\end{axis}
\end{tikzpicture}
\qquad\quad  
\begin{tikzpicture}
\begin{axis}[footnotesize,
view/h=50,
xlabel={\footnotesize $x$}, ylabel={\footnotesize $y$},  zlabel={\footnotesize $\Im S(z)\phantom{n}$}, 
every axis z label/.style=
{at={(zticklabel cs:0.5)},rotate=0,anchor=center},
every axis x label/.style=
{at={(xticklabel cs:0.5,3pt)},rotate=0,anchor=center},
every axis y label/.style=
{at={(yticklabel cs:0.5,-1pt)},rotate=0,anchor=center},
]
\addplot3[surf,mesh/ordering=y varies,mesh/rows=41
]
table {im.dat};
\end{axis}
\end{tikzpicture}
\caption{Graphs of $|S(x+iy)|$ and $\Re S(x+iy)$ for $(x,y)\in[0,8]\times[-8,8]$, and of $\Im S(x+iy)$ for $(x,y)\in[0,16]\times[-8,8]$.}%
	\label{fig:}
\end{figure}
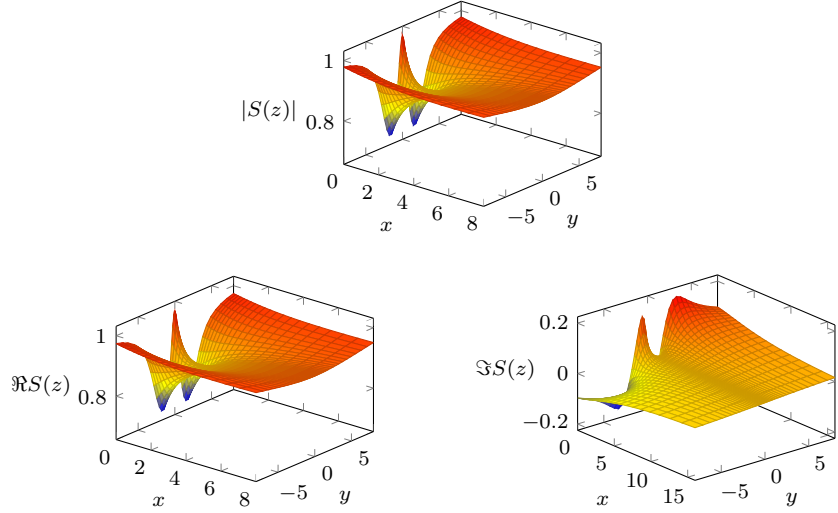



An immediate corollary of Theorem~\ref{th:} is 

\begin{corollary}\label{cor:R<}
For all $z\in H_+\setminus\{0\}$,  
\begin{equation}\label{eq:|R|<}
	|R(z)|<\big|z+\sqrt{2/\pi}\big|, 
\end{equation}
whereas $R(0)=\sqrt{2/\pi}$. 
Moreover, 
\begin{equation}\label{eq:asymp}
	R(z)\sim z \quad\text{as\quad $H_+\ni z\to\infty$. }
\end{equation}
\end{corollary}
\noindent
Here and in the sequel we use the following standard notation: $F\sim G$ meaning that $F/G\to1$; $F\lsim G$ or, equivalently, $G\gsim F$ meaning that $\limsup F/G\le1$; 
$F\ll G$ or, equivalently, $G\gg F$ meaning that $F/G\to0$; 
$F\O G$ or, equivalently, $G\OG F$ 
meaning that $\limsup|F/G|<\infty$, and $F\asymp G$ meaning that $F\O G\O F$. 


In turn, Corollary~\ref{cor:R<} yields the following bound on the $n$th derivative $R^{(n)}$ of the inverse Mills ratio  $R$. 

\begin{corollary}\label{cor:ders}
For any natural $n$ and any $z\in H_+$ with $x:=\Re z>0$, 
\begin{equation}\label{eq:R^{(n)}<}
	\big|R^{(n)}(z)\big|\le R^{(n)}_{\max}(z):=\frac{n!}{x^n}\,\sqrt{\big|z+\sqrt{2/\pi}\big|^2+x^2}. 
\end{equation}
Moreover, for each natural $n$,  
\begin{equation}\label{eq:R^{(n)}<<}
	\big|R^{(n)}(z)-\ii{n=1}\big|\ll\frac{|z|}{x^n}\quad\text{as\quad $z\in H_+$ and $x\to\infty$, }
\end{equation}
where $\ii\cdot$ stands for the indicator function. 
\end{corollary}

In applications of Corollary~\ref{cor:ders} to the calculation of sums of the form $s_{x_0,\de,N}:=\sum_{i=0}^{N-1}R(x_0+i\de)$ for natural $N$ and positive $x_0$ and $\de$ (see e.g.\ \cite{euler-maclaurin-alt}), of special interest is the case when $x>>n>>1$, which makes the bound $R^{(n)}_{\max}(x)$ small (here $x$ is real, as before). In such a case, the bound $R^{(n)}_{\max}(x)$ is optimal at least in the logarithmic sense -- which is the appropriate sense as far as the desired number of digits in the calculation of the sums $s_{x_0,\de,N}$ is concerned. Indeed, let us compare the bound $R^{(n)}_{\max}(x)$ on the $n$th derivative of the function $R$ with the $n$th derivative of the function $f(x)=x+1/x$, which is asymptotic to $R(x)$ as $x\to\infty$. We see that $\log R^{(n)}_{\max}(x)\sim-n\log\frac xn\sim\log|f^{(n)}(x)|$ if $x>>n>>1$. 

\section{Proofs}\label{proofs}

\begin{proof}[Proof of Lemma~\ref{lem:}]\ 

\textbf{\eqref{conj}} The conjugation symmetry property of the function $\vpi$ is obvious. That of $\G$ follows because 
$\bar{\G(z)}=\int_z^\infty\bar{\vpi(w)}\dd\bar w=\int_{\bar z}^\infty\vpi(w)\dd w=\G(\bar z)$ for all $z\in\C$.  
for all real $x$. 
 
\textbf{\eqref{G ne 0,lem}} Take any $z=x+iy\in H_+$ with $x:=\Re z$ and $y:=\Im z$. 
By part~\eqref{conj} of Lemma~\ref{lem:}, without loss of generality $y\ge0$. 
If $x=0$, then $\G(z)=\G(iy)=-\int_0^y\vpi(iv)\,i\dd v+\G(0)
=-i\int_0^y\frac1{\sqrt{2\pi}}\,e^{v^2/2}\dd v+\frac12$, 
so that $\Re\G(z)>0$ and $\Im\G(z)$ equals $-\Im z$ in sign. So, since $\vpi(iy)>0$ for real $y$, the statements made in part~\eqref{G ne 0,lem} follow, in the case $x=0$. 
Suppose now that $x>0$. Then, integrating from $z=x+iy$ to $\infty+0i$ along a curve with the image set $C:=\{w\in H_+\colon uv=xy, u\ge x\}$, where $u:=\Re w$ and $v:=\Im w$, one has 
\begin{equation}\label{eq:A-iB}
\frac{\G(z)}{\vpi(z)}=A-iB, 
\end{equation}
where 
\begin{equation}\label{eq:A,B}
	A:=A(z):=\int_x^\infty e^{a(u)-a(x)}\dd u\quad\text{and}\quad 
	B:=B(z):=\int_0^y e^{a(y)-a(v)}\dd v, 
\end{equation}
where 
\begin{equation}\label{eq:a}
	a(u):=a_{xy}(u):=\frac{x^2y^2}{2u^2}-\frac{u^2}2. 
\end{equation}
Now the statements made in part~\eqref{G ne 0,lem} follow from \eqref{eq:A-iB}, because $A(z)>0$ and $B(z)\ge0$ (for $y\ge0$), with $B(z)=0$ only if $\Im z=0$. 


\textbf{\eqref{lim}} The first equality in \eqref{eq:lim} is trivial. Let us prove the second equality there. Let $H_+\ni z=x+iy\to\infty$, where $x:=\Re z$ and $y:=\Im z$; that is, $x\ge0$, $y\in\R$, and $x^2+y^2\to\infty$. In view of the continuity of the function $S$ and its conjugation symmetry property, 
without loss of generality $x>0$ and $y>0$. 


Let $r$ denote the Mills ratio, so that $r=1/R=\G/\vpi$. It is well known and easy to prove using l'Hospital's rule that 
\begin{equation}\label{eq:r sim 1/x}
	r(x)\sim1/x\quad\text{as }x\to\infty.  
\end{equation}
Similarly, for the rescaled version $\tr$ of the Dawson function defined by the formula 
\begin{equation}
	\tr(y):=e^{-y^2/2}\int_0^y e^{u^2/2}\dd u
\end{equation}
for $y>0$, one has 
\begin{equation}\label{eq:tr sim 1/y}
	\tr(y)\sim1/y\quad\text{as }y\to\infty.  
\end{equation}

Recalling \eqref{eq:a}, note that 
$a'(u)=-x^2y^2/u^3-u\ge-(u+y^2/x)$ for $u\ge x$, whence $a(u)-a(x)\ge\frac12\,(x+y^2/x)^2-\frac12\,(u+y^2/x)^2$ and, by \eqref{eq:A,B}, 
\begin{equation}\label{eq:Alo}
\begin{aligned}
	A\ge
	&\int_x^\infty\exp\big\{\tfrac12\,(x+y^2/x)^2-\tfrac12\,(u+y^2/x)^2\big\}\dd u \\ =&\int_{x+y^2/x}^\infty\exp\big\{\tfrac12\,(x+y^2/x)^2-\tfrac12\,s^2\big\}\dd s \\ 	
	=&r(x+y^2/x)\sim\frac x{x^2+y^2}. 
\end{aligned}	
\end{equation}
Here we used the observation that $x+y^2/x=(x^2+y^2)/x\ge\sqrt{x^2+y^2}\to\infty$ and \eqref{eq:r sim 1/x}. 

Similarly, 
$-a'(v)=x^2y^2/v^3+v\ge v+x^2/y$ for $v\in(0,y]$, whence $a(y)-a(v)\le\frac12\,(v+x^2/y)^2-\frac12\,(y+x^2/y)^2$ and, by \eqref{eq:A,B}, 
\begin{equation}\label{eq:Bhi}
\begin{aligned}
	B\le
	&\int_0^y\exp\big\{\tfrac12\,(v+x^2/y)^2-\tfrac12\,(x+x^2/y)^2\big\}\dd v \\ 
	\le&\int_0^{y+x^2/y}\exp\big\{\tfrac12\,t^2-\tfrac12\,(y+x^2/y)^2\big\}\dd t \\ 
	=&\tr(y+x^2/y)\sim\frac y{x^2+y^2},  
\end{aligned}	
\end{equation}
in view of \eqref{eq:tr sim 1/y}. 


Next, fix any $c\in(0,1)$. Then for all $v\in[cy,y]$ one has $-a'(v)=x^2y^2/v^3+v\le v+x^2/(c^3y)$, whence $a(y)-a(v)\ge\frac12\,\big(v+x^2/(c^3y)\big)^2-\frac12\,\big(y+x^2/(c^3y)\big)^2$. So, 
\begin{align*}
	B&\ge
	\int_{cy}^y\exp\Big\{\frac12\,\Big(v+\frac{x^2}{c^3y}\Big)^2-\frac12\,\Big(y+\frac{x^2}{c^3y}\Big)^2\Big\}\dd v \\  	
=&\tr\Big(y+\frac{x^2}{c^3y}\Big)
	-\exp\Big\{-\frac12\Big(y+\frac{x^2}{c^3y}\Big)^2\Big\}
	\exp\Big\{\frac12\Big(cy+\frac{x^2}{c^3y}\Big)^2\Big\}\tr\Big(cy+\frac{x^2}{c^3y}\Big) \\ 
=&\tr\Big(y+\frac{x^2}{c^3y}\Big)
	-\exp\Big\{-\frac{1-c^2}2\,y^2-\frac{1-c}{c^3}\,x^2\Big\}\tr\Big(cy+\frac{x^2}{c^3y}\Big) \\ 
	\sim&\frac{c^3y}{x^2+c^3y^2};   
\end{align*}
here we again used \eqref{eq:tr sim 1/y} and the condition $x^2+y^2\to\infty$. 
Letting now $c\uparrow1$ and recalling \eqref{eq:Bhi}, we have 
\begin{equation}\label{eq:B sim}
	B\sim\frac y{x^2+y^2}\quad\text{as }x>0,\ y>0,\ x^2+y^2\to\infty. 
\end{equation}

Further, fix any real $k>1$. Then for all $u\in[x,kx]$ one has $a'(u)=-x^2y^2/u^3-u\le-\big(u+y^2/(k^3x)\big)$, whence $a(u)-a(x)\le-\frac12\,\big(u+y^2/(k^3x)\big)^2+\frac12\,\big(x+y^2/(k^3x)\big)^2$. So, 
\begin{align*}
	A&\le
	A_1(k)+A_2(k), 
\end{align*}
where 
\begin{align*}	
A_1(k):=	&
	\int_x^{kx}\exp\Big\{-\frac12\,\Big(u+\frac{y^2}{k^3x}\Big)^2+\frac12\,\Big(x+\frac{y^2}{k^3x}\Big)^2\Big\}\dd u \\ 
\le&\int_x^\infty\exp\Big\{-\frac12\,\Big(u+\frac{y^2}{k^3x}\Big)^2+\frac12\,\Big(x+\frac{y^2}{k^3x}\Big)^2\Big\}\dd u \\ 
=&r\Big(x+\frac{y^2}{k^3x}\Big)
	\sim\frac{k^3x}{k^3x^2+y^2},    
\end{align*}
in view of \eqref{eq:r sim 1/x} and the condition $x^2+y^2\to\infty$, and 
\begin{align}	
A_2(k):=	
	\int_{kx}^\infty e^{a(u)-a(x)}\dd u 
	&=\int_{kx}^\infty \exp\Big\{\frac{x^2y^2}{2u^2}-\frac{u^2}2-a(x)\Big\}\dd u \notag \\  
&\le \exp\Big\{\frac{y^2}{2k^2}-a(x)\Big\} \int_{kx}^\infty e^{-u^2/2}\dd u \notag \\  
&<\exp\Big\{\frac{y^2}{2k^2}-a(x)\Big\} e^{-k^2x^2/2}\frac1{kx} \notag \\  
&<\exp\Big\{-\frac{k^2-1}{2k^2}\,(x^2+y^2)\Big\}\frac1{kx}  \notag \\  
&\ll \frac x{x^2+y^2}  \quad\text{if $x^2(x^2+y^2)\OG1$,}\label{eq:Ahii2}   
\end{align}	
because then $\frac1{kx}\O x(x^2+y^2)$, while $\exp\big\{\kern-2pt-\frac{k^2-1}{2k^2}\,(x^2+y^2)\big\}
\ll 1/(x^2+y^2)^2$. \hfill

Letting now $k\downarrow1$ and recalling \eqref{eq:Alo}, we have 
\begin{equation}\label{eq:A sim}
	A\sim\frac x{x^2+y^2}\quad\text{as }x>0,\ y>0,\ x^2+y^2\to\infty,\ x^2(x^2+y^2)\OG1. 
\end{equation}

In view of \eqref{eq:R}, \eqref{eq:A-iB}, \eqref{eq:B sim}, and \eqref{eq:A sim}, 
\begin{equation}\label{eq:R sim}
	R(z)=\frac1{A-iB}\sim\frac{x^2+y^2}{x-iy}=z 
\end{equation} 
as $x>0$, $y>0$,\ $x^2+y^2\to\infty$,\ and $x^2(x^2+y^2)\OG1$. 

Suppose now the condition $x^2(x^2+y^2)\OG1$ does not hold, so that without loss of generality 
$x^2(x^2+y^2)\to0$ (while still $x>0,\ y>0,\ x^2+y^2\to\infty$). Then $x\to0$, $xy\to0$, $y
\to\infty$,  $z\sim iy$, and $\vpi(x+iy)=\vpi(iy)\exp\{-x^2/2-ixy\}\sim\vpi(iy)\to\infty$. 
So,  
\begin{equation}\label{eq:I_1,I_2}
\G(z)=-I_1-I_2+\G(0)=-I_1-I_2+1/2, 	
\end{equation}
where 
\begin{equation}
	I_1:=\int_0^x\vpi(u+iy)\dd u\quad\text{and}\quad
	I_2:=i\int_0^y\vpi(iv)\dd v. 
\end{equation}
At that, 
\begin{equation}
	|I_1|\le x\vpi(iy)\ll\frac{\vpi(iy)}{\sqrt{x^2+y^2}}\sim|\vpi(z)/z|
\end{equation}
and, by \eqref{eq:tr sim 1/y},  
\begin{equation}
	I_2\sim i\vpi(iy)/y\sim -\vpi(z)/z. 
\end{equation}
It also follows that $|I_2|\to\infty$. 

So, in view of \eqref{eq:R} and \eqref{eq:I_1,I_2}, asymptotic equivalence \eqref{eq:R sim} holds whenever $x>0$, $y>0$, and $x^2+y^2\to\infty$.  
This completes the proof of part~\eqref{lim} of Lemma~\ref{lem:}.  

\textbf{\eqref{extr}} 
Let 
\begin{equation}\label{eq:s}
s(y):=|S(iy)|^2=\frac{f(y)}{g(y)},
\end{equation}
where 
\begin{equation}\label{eq:f,g}
f(y):=\frac{\vpi(iy)^2}{y^2 + 2/\pi},\quad 
g(y):=\frac14+E(y)^2,\quad E(y):=\int_0^y\vpi(iv)\dd v. 
\end{equation}
Here and in the rest of the proof of Lemma~\ref{lem:}, $y$ stands for an arbitrary nonnegative real number. 
Consider first two ``derivative ratios'' for the ratio $s=f/g$: 
\begin{equation}
	s_1:=\frac{f'}{g'}=
	\frac{f_1}{g_1}\quad\text{and}\quad s_2:=\frac{f'_1}{g'_1}, 
\end{equation}
where $g_1:=E$ and $f_1:=s_1 E$. Then $s_2$ is a rational function, that is, the ratio of two polynomials (each of degree $6$). So, it is straightforward to find that, for some algebraic numbers $y_{21}$ and $y_{22}$ such that $0<y_{21}<y_{22}$, the function $s_2$ is (strictly) increasing on the interval $[0,y_{21}]$, decreasing on $[y_{21},y_{22}]$, and increasing on $[y_{22},\infty)$; in fact, $y_{21}=0.685\ldots$ and $y_{22}=1.407\ldots$.  

Using the l'Hospital rule for limits, one easily finds that $\lim_{y\downarrow0}s'_1(y)/y=\frac16\, (2 - 4 \pi + 3 \pi^2)>0$. So, by the general l'Hospital-type rule for monotonicity, given e.g.\ by line~1 of Table~1.1 in \cite{pin06}, 
the function $s_1$ is increasing on the interval $[0,y_{21}]$. Next, 
$s'_1(y_{22})=0.054\ldots>0$; so, by lines~2 and 1 of 
Table~1.1 in \cite{pin06}, $s_1$ is increasing on the intervals $[y_{21},y_{22}]$ and $[y_{22},\infty)$ as well. Thus, the first ``derivative ratio'' $s_1$ for the ratio $s$ is increasing on the entire interval $[0,\infty)$. 

The values of the function $s$ at the points $0$, $1$, and $3$ are $1$, $0.553\ldots$, and $0.670\ldots$, respectively, so that $s(0)>s(1)<s(3)$. Using again line~1 of 
Table~1.1 in \cite{pin06}, we conclude that $s$ is decreasing-increasing on $[0,\infty)$; that is, there is a uniquely determined number $y_*\in(0,\infty)$ such that $s$ decreases on $[0,y_*]$ and increases on $[y_*,\infty)$. 
In other words, $|S(iy)|=\sqrt{s(y)}$ decreases in $y\in[0,y_*]$ and increases in $y\in[y_*,\infty)$. 
Let $y_{01}:=16267/10000$ and $y_{02} := 16268/10000$. Then $s'(y_{01})<0<s'(y_{02})$, whence  $0<y_{01}<y_*<y_{02}$ and $y_*=1.6267\ldots$. 

So, in view of \eqref{eq:s} and \eqref{eq:f,g}, 
\begin{equation}
	\min_{y\ge0}|S(iy)|=|S(iy_*)|>\frac{\vpi(iy_{01})}{\sqrt{(y_{02}^2+2/\pi)\big(1/4+E(y_{02})^2\big)}}=0.6861\ldots; 
\end{equation}
here we used the obvious fact that the expressions $\vpi(iy)$, $y^2+2/\pi$, and $1/4+E(y)^2$ are each increasing in $y\ge0$. 
On the other hand, $\min_{y\ge0}|S(iy)|=|S(iy_*)|\le|S(iy_{02})|=0.6862\ldots$. Thus, $|S(iy_*)|=0.686\ldots$. 

This concludes the proof of part~\eqref{extr} of Lemma~\ref{lem:} and thereby that of the entire lemma. 
\end{proof}

\begin{proof}[Proof of Proposition~\ref{prop:x>0}]
This proof is similar to, and even significantly simpler than, the proof of part~\eqref{extr} of Lemma~\ref{lem:}. Indeed, let here $s$ be the restriction of the function $S$ to $[0,\infty)$, so that 
$s=f/g$, where $f(x):=\vpi(x)/(x + \sqrt{2/\pi})$ and $g(x):=\Psi(x)$ for $x\ge0$. Then the ``derivative ratio'' $f'/g'$ is a (rather simple) rational function, 
which decreases on the interval $[0,x_*]$ \big(with $x_*=(\pi-1)\sqrt{2/\pi} = 1.7087\ldots$, as in the statement of Proposition~\ref{prop:x>0}\big) and increases on the interval $[x_*,\infty)$. Also, clearly $f(\infty-)=g(\infty-)=0$. Moreover, $S(x_*)=0.844\ldots<1=S(0)=S(\infty-)$. 
Now Proposition~\ref{prop:x>0} follows immediately by the derived special l'Hospital-type rule for monotonicity given in the last line of Table 4.1 in \cite{pin06}.   
\end{proof}

\begin{proof}[Proof of Corollary~\ref{cor:ders}]\ 
Take indeed any natural $n$ and any $z\in H_+$ with $x=\Re z>0$. For any real $\vp>0$, let $C_{z;\vp}$ denote the circle of radius $\vp$ centered at the point $z$, traced out counterclockwise. By the Cauchy integral formula, 
\begin{equation}\label{eq:R^{(n)}=}
	R^{(n)}(z)=\frac{n!}{2\pi i}\,\int_{C_{z;x}}\frac{R(\zeta)\dd \zeta}{(\zeta-z)^{n+1}}
	=\frac{n!}{2\pi x^n}\,\int_0^{2\pi}R(z+x e^{it})e^{-int}\dd t. 
\end{equation}
So, by \eqref{eq:|R|<}, 
\begin{multline}\label{eq:=J}
	\frac{2\pi x^n}{n!}\,|R^{(n)}(z)|\le\int_0^{2\pi}|z+\sqrt{2/\pi}+x e^{it}|\dd t \\ 
	=\int_0^{2\pi}\sqrt{a+b\cos(t-\theta)}\dd t=\int_0^{2\pi}\sqrt{a+b\cos t}\dd t=:J(a,b), 
\end{multline}
where 
\begin{equation}\label{eq:a,b,th}
	a:=\big|z+\sqrt{2/\pi}\big|^2+x^2,\quad b:=2x\big|z+\sqrt{2/\pi}\big|,\quad \theta:=\arctan\frac y{x+\sqrt{2/\pi}}. 
\end{equation}
The integral $J(a,b)$ is an elliptic one. It admits a simple upper bound, $J(a,0)=2\pi\sqrt a$, which is rather accurate \big(not exceeding $\frac{10}9\,J(a,b)$\big), for any $a$ and $b$ such that $a>0$ and $b\in[0,a]$. This follows because $J(a,b)$ is obviously concave in $b\in[0,a]$, with the partial derivative in $b$ at $b=0$ equal $0$, so that $J(a,b)$ is nonincreasing in $b\in[0,a]$, from $J(a,0)=2\pi\sqrt a$ to $J(a,a)=4\sqrt{2a}>\frac9{10}\,J(a,0)$. 
%
Now \eqref{eq:R^{(n)}<} follows immediately from \eqref{eq:=J} and \eqref{eq:a,b,th}.  


Clearly, the identities in \eqref{eq:R^{(n)}=} hold with $x/2$ in place of $x$. Let now $t\in[0,2\pi]$, $z\in H_+$, and $x=\Re z\to\infty$. 
Then $2|z|\ge\big|z+\frac x2\,e^{it}\big|\ge|z|/2\to\infty$.
So, in view of \eqref{eq:asymp}, 
\begin{multline}
	\frac{2\pi}{n!}\,\Big(\frac x2\Big)^nR^{(n)}(z)
	=\int_0^{2\pi}R\Big(z+\frac x2\,e^{it}\Big)e^{-int}\dd t  
	=\int_0^{2\pi}\Big(z+\frac x2\,e^{it}+o(|z|)\Big)e^{-int}\dd t \\ 
	=2\pi\,\frac x2\,\ii{n=1}+\int_0^{2\pi}o(|z|)e^{-int}\dd t
	=2\pi\,\frac x2\,\ii{n=1}+o(|z|),  
\end{multline}
so that \eqref{eq:R^{(n)}<<} follows as well. 
\end{proof}

\bibliographystyle{elsarticle-num} 

\bibliography{C:/Users/ipinelis/Dropbox/mtu/bib_files/citations12.13.12}

\def\cprime{$'$} \def\polhk#1{\setbox0=\hbox{#1}{\ooalign{\hidewidth
  \lower1.5ex\hbox{`}\hidewidth\crcr\unhbox0}}}
  \def\polhk#1{\setbox0=\hbox{#1}{\ooalign{\hidewidth
  \lower1.5ex\hbox{`}\hidewidth\crcr\unhbox0}}}
  \def\polhk#1{\setbox0=\hbox{#1}{\ooalign{\hidewidth
  \lower1.5ex\hbox{`}\hidewidth\crcr\unhbox0}}} \def\cprime{$'$}
  \def\polhk#1{\setbox0=\hbox{#1}{\ooalign{\hidewidth
  \lower1.5ex\hbox{`}\hidewidth\crcr\unhbox0}}}
  \def\polhk#1{\setbox0=\hbox{#1}{\ooalign{\hidewidth
  \lower1.5ex\hbox{`}\hidewidth\crcr\unhbox0}}} \def\cprime{$'$}
  \def\cprime{$'$}
\begin{thebibliography}{1}
\expandafter\ifx\csname url\endcsname\relax
  \def\url#1{\texttt{#1}}\fi
\expandafter\ifx\csname urlprefix\endcsname\relax\def\urlprefix{URL }\fi
\expandafter\ifx\csname href\endcsname\relax
  \def\href#1#2{#2} \def\path#1{#1}\fi

\bibitem{ahlfors79}
L.~V. Ahlfors, Complex analysis, 3rd Edition, McGraw-Hill Book Co., New York,
  1979.

\bibitem{guillemin-sternberg}
V.~Guillemin, S.~Sternberg, Geometric asymptotics, American Mathematical
  Society, Providence, R.I., 1977, {M}athematical {S}urveys, No. 14.

\bibitem{shenton}
L.~R. Shenton, Inequalities for the normal integral including a new continued
  fraction, Biometrika 41 (1954) 177--189.

\bibitem{baricz08}
{\'A}.~Baricz, \href{http://dx.doi.org/10.1016/j.jmaa.2007.09.063}{Mills'
  ratio: monotonicity patterns and functional inequalities}, J. Math. Anal.
  Appl. 340~(2) (2008) 1362--1370.
\newblock \href {http://dx.doi.org/10.1016/j.jmaa.2007.09.063}
  {\path{doi:10.1016/j.jmaa.2007.09.063}}.
\newline\urlprefix\url{http://dx.doi.org/10.1016/j.jmaa.2007.09.063}

\bibitem{euler-maclaurin-alt}
I.~Pinelis, An alternative to the {E}uler--{M}aclaurin formula: {A}pproximating
  sums by integrals only, \url{http://arxiv.org/abs/1511.03247} (2015).

\bibitem{pin06}
I.~Pinelis, On l'{H}ospital-type rules for monotonicity, JIPAM. J. Inequal.
  Pure Appl. Math. 7~(2) (2006) Article 40, 19 pp. (electronic).

\end{thebibliography}


\end{document}